\documentclass[12pt]{amsart}

\usepackage[english]{babel}

\usepackage[utf8]{inputenc}

\usepackage{microtype,amsmath,amsthm,amssymb}
\usepackage{parskip}
\usepackage{geometry}
\usepackage[hidelinks]{hyperref}
\author{Sergei Konyagin}
\address{Steklov Institute of Mathematics \\ 8 Gubkin Street, Moscow, 119991, Russia}
\email{konyagin@mi-ras.ru}

\author{Paul Pollack}
\address{Department of Mathematics \\ University of Georgia \\ Athens, GA 30602, USA}
\email{pollack@uga.edu}

\renewcommand\subset\subseteq

\renewcommand{\pod}[1]{\allowbreak\mathchoice
  {\if@display \mkern 18mu\else \mkern 8mu\fi (#1)}
  {\if@display \mkern 18mu\else \mkern 8mu\fi (#1)}
  {\mkern4mu(#1)}
  {\mkern4mu(#1)}
}
\usepackage{graphicx}
\DeclareMathAlphabet{\curly}{U}{rsfs}{m}{n}

\newcommand\Z{\mathbb{Z}}
\newcommand\Q{\mathbb{Q}}
\newtheorem{thm}{Theorem}[section]

\newtheorem{prop}[thm]{Proposition}
\newtheorem{lem}[thm]{Lemma}

\newtheorem{hyp}[thm]{Hypothesis}

\theoremstyle{remark}
\newtheorem{rmk}[thm]{Remark}

\newcommand\ord{\mathrm{ord}}

\newcommand{\genlegendre}[4]{%
  \genfrac{(}{)}{}{#1}{#3}{#4}%
  \if\relax\detokenize{#2}\relax\else_{\!#2}\fi
}

\newcommand{\leg}[3][]{\genlegendre{}{#1}{#2}{#3}}

\title{A problem in comparative order theory}

\begin{document}
\begin{abstract} Write $\ord_p(\cdot)$ for the multiplicative order in $\mathbb{F}_p^{\times}$. Recently, Matthew Just and the second author investigated the problem of classifying pairs $\alpha, \beta \in \mathbb{Q}^{\times}\setminus\{\pm 1\}$ for which $\ord_p(\alpha) > \ord_p(\beta)$ holds for infinitely many primes $p$. They called such pairs \textsf{order-dominant}.
We describe an easily-checkable sufficient condition for $\alpha,\beta$ to be order-dominant. Via the large sieve, we show that almost all integer pairs $\alpha,\beta$ satisfy our condition, with a power savings on the size of the exceptional set.
\end{abstract}

\maketitle

\section{Introduction}
For each rational prime $p$, we write $v_p(\cdot)$ for the $p$-adic valuation. If $\epsilon\in \Q_p$ with $v_p(\epsilon)=0$, we let $\ord_p(\epsilon)$ denote the multiplicative order of $\epsilon$ in the residue ring $\Z_p/p\Z_p \cong \Z/p\Z = \mathbb{F}_p$.
Now fix $\alpha, \beta \in \mathbb{Q}^{\times}\setminus\{\pm 1\}$. For all but finitely many primes $p$, both $\alpha$ and $\beta$ are units in $\Z_p$, and so it is sensible to ask how $\ord_p(\alpha)$ and $\ord_p(\beta)$ compare, as $p$ varies within the set of primes.

From results of Corrales-Rodriga\~{n}ez and Schoof \cite{CS97}, if $\ord_p(\alpha)=\ord_p(\beta)$ for all but finitely many $p$, then $\alpha=\beta^{\pm 1}$.
(Schinzel \cite{schinzel60} had earlier obtained results in the same direction as \cite{CS97} but less general.)
In \cite{SW92}, Schinzel and W\'{o}jcik show that whenever $\alpha, \beta \in\Q^{\times}\setminus \{\pm 1\}$, there are infinitely many primes $p$ with $v_p(\alpha)=v_p(\beta)=0$ and $\ord_p(\alpha)=\ord_p(\beta)$. (Earlier, special cases of this result had been worked out --- but not published --- by J.S. Wilson, J.G. Thompson, and J.W.S. Cassels.) Inspired by these works, Just and the second author \cite{JP21} recently investigated the following related question: For which pairs $\alpha,\beta$ are there infinitely many $p$ with $v_p(\alpha)=v_p(\beta)=0$ and $\ord_p(\alpha) > \ord_p(\beta)$? They termed such pairs \textsf{order-dominant}.

Under GRH, one can show that $\alpha,\beta$ is order-dominant as long as $\alpha$ is not a power of $\beta$. (When $\alpha$ and $\beta$ are multiplicatively independent, a property stronger than order-dominance is proved, subject to GRH, by Järviniemi; see \cite[Theorem 1.4]{j20}. The remaining cases are easier and do not need GRH.) In \cite{JP21}, Just and Pollack unconditionally establish order-dominance for several families of integer pairs $A,B$. A typical result of \cite{JP21} is as follows: If $A, B$ are odd positive integers, and the Jacobi symbol $\leg{-B(1-B)}{A}=-1$  or $\leg{1-B}{A}=-1$, then $A,B$ is order dominant. Key to the proof of this and many of the other theorems of \cite{JP21} is the following observation (implicit in work) of Banaszak \cite{banaszak98}.

Let $\alpha,\beta \in \Q^{\times}\setminus\{\pm 1\}$, and let $p$ be an odd prime. Suppose that (a) $\alpha,\beta$ are units in $\Z_p$, (b) $v_p( \alpha^{2k}-\beta) \ge 1$ for some integer $k$, and (c) the Legendre symbol $\leg{\alpha}{p}=-1$. Condition (b) implies that $\beta$ is a square mod $p$ and that we have a containment of cyclic subgroups, $\langle \beta\bmod{p}\rangle \subset \langle \alpha\bmod{p}\rangle$. Since $\beta$ is a square mod $p$ while $\alpha$ is not (by (c)), $\alpha\bmod{p}\notin \langle \beta\bmod{p}\rangle$. Thus, $\langle \beta\bmod{p}\rangle \subsetneq \langle \alpha\bmod{p}\rangle$, and so $\ord_p(\beta)< \ord_p(\alpha)$. (In fact, we get slightly more from this argument: $\ord_{p}{(\alpha)}/\ord_p{(\beta)}$ is a positive even integer.) Thus, if infinitely many primes $p$ satisfy (a)--(c), then $\alpha,\beta$ is an order-dominant pair.

Let us call a pair $\alpha,\beta$ satisfying (a)--(c) for infinitely many $p$ a \textsf{special pair}. Our main result is an easily-checkable criterion for $\alpha,\beta$ to be special.
\begin{thm}\label{thm:extension} Let $\alpha, \beta \in \Q^{\times}\setminus \{\pm 1\}$. Suppose there is an odd prime $q$ for which
\begin{enumerate}
    \item[(i)] $v_{q}(\alpha)=v_{q}(\beta)=0$,
    \item[(ii)] the Legendre symbol $\leg{\alpha}{q}=-1$,
    \item[(iii)] there is an integer $k$ with $v_{q}(\alpha^{2k}-\beta) \ge v_{q}(\alpha^{q-1}-1)$.
\end{enumerate}
Then $\alpha,\beta$ is a special pair.
\end{thm}

The authors of \cite{JP21} report that they were unable to establish the order-dominance of the pair $17, 2$. That $17,2$ is order-dominant (indeed, special) is an immediate consequence of Theorem \ref{thm:extension}: Take $q=7$ and $k=1$.

Probably $\alpha, \beta$ is special as long as $\alpha$ is not a rational square and $\Q(\sqrt{\alpha})\ne \Q(\sqrt{\beta})$. These conditions are easily seen to be necessary. A simple computer program implementing the criterion of Theorem \ref{thm:extension} verifies that these conditions are also sufficient for all integer pairs $\alpha,\beta$ with $1 < |\alpha|, |\beta| \le 1000$.  We do not know how to prove they are sufficient in general. 
Nevertheless, using the large sieve we can show that being ``special'' is actually rather ordinary: Almost all integer pairs are special.

\begin{thm}\label{thm:almostall} For all large positive integers $N$, the number of non-special pairs of integers $\alpha,\beta$ with $1<|\alpha|,|\beta| \le N$ is $O(N^{3/2}\log{N})$.
\end{thm}

Since $\alpha,\beta$ is non-special whenever $\alpha$ is a square, the upper bound of Theorem \ref{thm:almostall} is sharp up to the log factor.

\section{Preliminaries}
The following result is often referred to as the lemma on ``Lifting the Exponent''. We choose a formulation over $\mathbb{Z}_p$, but the proof is the same as over $\Z$ (see, for example, \cite[Lemma 2.1.22, Corollary 2.1.23, p.\ 22]{cohen07}).

\begin{lem}\label{lem:LTE} Let $p$ be a prime number, and let $A \in \Z_p$  with $v_p(A-1) \ge 1$. Then
\[ v_p(A^n-1) \ge v_p(A-1) + v_p(n) \]
for every integer $n$. Furthermore, equality holds whenever $p$ is odd.
\end{lem}

\begin{lem}\label{lem:surjective} Let $p$ be an odd prime. Let $A \in \Z_p$ with $v_p(A-1) = t \ge 1$. For each integer $T\ge t$, and each $B \in \Z_p$ with $B\equiv 1\pmod{p^t}$, there is an integer $n$ with
\[ A^n \equiv B\pmod{p^T}. \]
\end{lem}
\begin{proof} Since $v_p(A-1)\ge 1$, we have $v_p(A)=0$, so that $A$ is invertible in $\Z_p$. For every pair of integers $n$ and $n'$,
\[ A^n \equiv A^{n'} \pmod{p^{T}} \Longleftrightarrow A^{n-n'} \equiv 1\pmod{p^T} \Longleftrightarrow n \equiv n' \pmod{p^{T-t}}, \]
using Lemma \ref{lem:LTE} for the final equivalence. It follows that $A$ mod $p^T$ generates a subgroup of $(\Z_p/p^T \Z_p)^\times$ of order $p^{T-t}$. Since each power of $A$ is congruent to $1$ mod $p^t$, and there are only $p^{T-t}$ elements mod $p^T$ congruent to $1 \bmod p^t$, the result follows.
\end{proof}

The proof of Theorem \ref{thm:extension} relies crucially on the quadratic reciprocity law. Let $d$ be a nonzero integer. If $\gamma \in \mathbb{Q}^{\times}$ and $v_p(\gamma)=0$ for all $p\mid 2d$, we define the generalized Jacobi symbol
\[ \leg{d}{\gamma} := \prod_{p:~v_p(\gamma)\ne 0} \leg{d}{p}^{v_p(\gamma)},   \]
where each $\leg{d}{p}$ is a Legendre symbol.

\begin{lem}\label{lem:QR} Let $d$ be a nonzero integer, and let $\eta \in \Q_{>0}$. Suppose that
\[ v_p(\eta-1) \ge 1 \qquad\text{for all odd primes $p\mid d$} \]
and that
\[ v_2(\eta-1) \ge 3. \]
Then $\leg{d}{\eta}=1$.
\end{lem}

\begin{proof} This is a consequence of quadratic reciprocity, viewed as a special case of Artin's reciprocity law. A down-to-earth proof is as follows: We may write $\eta=P/Q$ where $P, Q$ are positive integers coprime to each other and each coprime to $2d$. Then $\leg{d}{\eta} = \leg{d}{P}\leg{d}{Q} = \leg{\Delta}{P}\leg{\Delta}{Q}$, where $\Delta$ is the discriminant of $\Q(\sqrt{d})$. The conditions on the valuations of $\eta-1$ imply that $P\equiv Q\pmod{\Delta}$. The desired result follows from recalling that the Kronecker symbol $\leg{\Delta}{\cdot}$ is a Dirichlet character modulo $|\Delta|$ (see \cite[Theorem 2.2.9, p.\ 38]{cohen07}), so that $\leg{\Delta}{P}=\leg{\Delta}{Q}$ and $\leg{\Delta}{P}\leg{\Delta}{Q} = (\pm 1)^2 = 1$.
\end{proof}

Finally, we recall Bombieri's form of the arithmetic large sieve (see \cite[Theorem 1]{bombieri65}).

\begin{prop}\label{prop:largesieve} Let $N$ be a positive integer, and let $\mathcal{N}$ be a subset of $[-N,N]$. Set $Z:=\#\mathcal{N}$. For each prime $p \le \sqrt{N}$ and each integer $h$, let $Z(p,h) = \#\{n \in \mathcal{N}: n\equiv h\pmod{p}\}$. Then
\[ \sum_{p \le \sqrt{N}} p \sum_{h\bmod{p}}\left(Z(p,h) - \frac{Z}{p}\right)^2 \ll NZ.\]
\end{prop}

\section{Proof of Theorem \ref{thm:extension}}
By replacing $\alpha$ with $1/\alpha$, we can (and will) assume that $|\alpha| > 1$. Our task is to show that the set
\[ \mathcal{P} = \{\text{odd primes }p: v_p(\alpha)=v_p(\beta)=0, \leg{\alpha}{p}=-1, \beta\bmod{p}\in \langle\alpha^2\bmod{p}\rangle\} \]
is infinite.

We suppose for a contradiction that $\mathcal{P}$ is finite. Note that $\mathcal{P}$ is nonempty since $q \in \mathcal{P}$.

Since the details of the proof are somewhat intricate, we first give the basic idea. Let $d$ be the unique squarefree integer differing from $\alpha$ by the square of a rational number. We will consider the generalized Jacobi symbol $\leg{d}{\eta}$, where $\eta:=\frac{\alpha^{2K'}-\beta}{\alpha^{2K}-\beta}$ for well-chosen integers $K$ and $K'$. On the one hand, the reciprocity law in the form of Lemma \ref{lem:QR} will imply that $\leg{d}{\eta}=1$. On the other hand, by controlling the power of $q$ appearing in the factorizations of $\alpha^{2K'}-\beta$ and $\alpha^{2K}-\beta$, we will be able, using $\leg{d}{q} = \leg{\alpha}{q}=-1$, to force  $\leg{d}{\eta}=-1$.

Now down to business. Put $e := v_q(\alpha^{q-1}-1)$. Setting
\[ \ell = \ord_q(\alpha^2), \]
we observe that $\ell \mid \frac{q-1}{2}$. By Lemma \ref{lem:LTE}, \begin{align*} e &= v_q((\alpha^{2\ell})^{(q-1)/2\ell}-1) \\
&= v_q(\alpha^{2\ell}-1).
\end{align*}
By assumption (iii), we may choose an integer $k$ with
\[ v_q(\alpha^{2k}-\beta)\ge e. \]
Adjusting $k$ by a multiple of $\phi(q^{e})$, we can assume that $k$ is large. Concretely, we assume that
\[ \alpha^{2k}-\beta > 0 \]
and that
\begin{equation}\label{eq:klarge} k \ge 3 + \max_{p} |v_p(\beta)|, \end{equation}
where the maximum is taken over all primes $p$.

Let
\[ e' = v_q(\alpha^{2k}-\beta), \]
and define
\[ M = q^{e'} \prod_{\substack{p \in \mathcal{P} \\ p \ne q}} (p-1) \prod_{p:~v_p(\alpha)\ne 0} (p-1).  \]
Put
\[ e'' = v_q(M), \]
so that $e'' \ge e' \ge e$. Viewing the rational numbers with denominator prime to $q$ as embedded in $\Z_q$, we may write
\[ \alpha^{2k} = \beta(1 + q^{e'} \epsilon), \]
where $\epsilon \in \Z_q^{\times}$. Choose a positive integer $n_0$ such that
\[ (\alpha^{2\ell})^{n_0} \equiv \frac{1+q^{e + e''}}{1+q^{e'}\epsilon} \pmod{q^{e+e''+1}}; \]
this is possible by Lemma \ref{lem:surjective} (with $A = \alpha^{2\ell})$, since $v_q(\alpha^{2\ell}-1)=e$ and $\frac{1+q^{e+e''}}{1+q^{e'}\epsilon} \equiv 1\pmod{q^{e}}$. Then
\[ v_q(\alpha^{2(k+\ell n_0)}-\beta) = e+e''. \]
We set $K=k+\ell n_0$; this is the integer $K$ described in  our initial proof summary. We proceed to choose $K'$.

For $p\in \mathcal{P}$, $p\ne q$, we let $f_p$ be the integer determined by the factorization
\[ \alpha^{2(k+\ell n_0)}-\beta = \Bigg(q^{e+e''} \prod_{\substack{p \in \mathcal{P} \\ p \ne q}} p^{f_p}\Bigg)R, \]
where $v_p(R)=0$ for all $p \in \mathcal{P}$. If $v_2(\alpha)=v_2(\beta)=0$, we further define
\[ f_2 = v_2(\alpha^{2(k+\ell n_0)}-\beta). \]
Let
\[ N = 2M \bigg(\prod_{\substack{p \in \mathcal{P} \\ p \ne q}}p^{f_p} \bigg)\cdot \widetilde{2^{f_2}},\]
where the tilde indicates that the final term is included only when $v_2(\alpha)=v_2(\beta)=0$. Then $v_q(N) = v_q(M) = e''$, so that by Lemma \ref{lem:LTE},
\begin{align*}
v_q(\alpha^{2\ell N}-1) &= v_q(\alpha^{2\ell}-1) + v_q(N)  \\
&= e + e''.
\end{align*}
Now writing $$ \alpha^{2(k+\ell n_0)} = \beta (1+q^{e+e''}\gamma),$$ where $\gamma \in \Z_{q}^{\times}$, we use Lemma \ref{lem:surjective} (with $A=\alpha^{2\ell N}$) to choose a positive integer $n_1$ such that
\[ \alpha^{2\ell N n_1} \equiv \frac{1 + q^{e+e''+1}}{1+q^{e+e''}\gamma} \pmod{q^{e+e''+2}}. \]
(This is possible, since the right-hand side of the congruence is $1$ mod $q^{e+e''}$.) Then
\[ v_q(\alpha^{2(k+\ell n_0 + \ell N n_1)}-\beta) = e+e''+1.\]
We take $K':= k+\ell n_0 + \ell N n_1$.

We proceed to compare the factorizations of $\alpha^{2K}-\beta$ and $\alpha^{2K'}-\beta$. Let $p\in \mathcal{P}\setminus\{q\}$, allowing also $p=2$ if $v_2(\alpha)=v_2(\beta)=0$. By construction, $\phi(p^{f_p+1}) \mid N$. Thus, $\alpha^{2\ell N n_1} \equiv 1\pmod{p^{f_p+1}}$, and so
\begin{align*}
v_p(\alpha^{2(k+\ell n_0 + \ell N n_1)}-\beta) = v_p((\alpha^{2 \ell N n_1} - 1)\alpha^{2(k+\ell n_0)} + (\alpha^{2(k+\ell n_0)}-\beta)) = v_p(\alpha^{2(k+\ell n_0)}-\beta)
\end{align*}
by the ultrametric inequality, since
\[ f_p = v_p(\alpha^{2(k+\ell n_0)}-\beta) < v_p(\alpha^{2 \ell N n_1} - 1). \]
It follows that the ratio
\begin{equation}\label{eq:ratioexpr} \frac{\alpha^{2K'}-\beta}{\alpha^{2K}-\beta} =  \frac{\alpha^{2(k+\ell n_0 + \ell N n_1)}-\beta}{\alpha^{2(k+\ell n_0)} -\beta} = q S, \end{equation}
where $v_p(S)=0$ for all $p \in \mathcal{P}$, $p\ne q$, and also $v_2(S)=0$ when $v_2(\alpha)=v_2(\beta)=0$.

Let us examine the factorization of $S$. Let $p$ be a prime, and suppose to start with that $v_p(\alpha)< 0$. By the ultrametric inequality and \eqref{eq:klarge}, $v_p(\alpha^{2K'}-\beta) = 2K' v_p(\alpha)$ while $v_p(\alpha^{2K}-\beta) = 2K v_p(\alpha)$. Hence,
\[ v_p(S) = 2(K'-K) v_p(\alpha) = 2\ell N n_1 v_p(\alpha). \]
If instead $v_p(\alpha) > 0$, then $v_p(\alpha^{2K'}-\beta) = v_p(\beta) = v_p(\alpha^{2K}-\beta)$, so that $v_p(S)=0$. Finally, suppose that $v_p(\alpha)=0$. If $v_p(\beta) > 0$, then $v_p(\alpha^{2K'}-\beta) = 0 = v_p(\alpha^{2K}-\beta)$, while if $v_p(\beta) < 0$, then $v_p(\alpha^{2K'}-\beta) = v_p(\beta) = v_p(\alpha^{2K}-\beta)$; in both cases, $v_p(S)=0$. Hence, if $v_p(\alpha)=0$ and $v_p(S) \ne 0$, then $v_p(\beta)=0$. If $v_p(\alpha)=v_p(\beta)=0$ and $v_p(S) \ne 0$, we know from the last paragraph that $p$ is odd and not in $\mathcal{P}$. Since $v_p(S)\ne 0$, it must be that $v_p(\alpha^{2K'}-\beta) > 0$ or $v_p(\alpha^{2K}-\beta) > 0$; in either case, $\beta\bmod{p}$ lands in the multiplicative group generated by $\alpha^2\bmod{p}$. So in order to have $p\notin \mathcal{P}$, it must be that $\leg{\alpha}{p}=1$. We conclude (since $S>0$) that $S$ admits a factorization
\[ S =  \prod_{p:~v_p(\alpha) < 0} p^{2\ell N n_1 v_p(\alpha)} \prod_{\substack{p\text{ odd} \\ v_p(\alpha)=v_p(\beta)=0 \\ \leg{\alpha}{p}=1}} p^{v_p(S)}. \]

Comparing the last display with \eqref{eq:ratioexpr}, we find upon rearranging that
\[ \frac{(\alpha \prod_{p:~v_p(\alpha)<0} p^{-v_p(\alpha)})^{2K'}-\beta (\prod_{p:~v_p(\alpha)<0} p^{-v_p(\alpha)})^{2K'}}{(\alpha \prod_{p:~v_p(\alpha)<0} p^{-v_p(\alpha)})^{2K}-\beta (\prod_{p:~v_p(\alpha)<0} p^{-v_p(\alpha)})^{2K}} = q \prod_{\substack{p\text{ odd} \\ v_p(\alpha)=v_p(\beta)=0 \\ \leg{\alpha}{p}=1}} p^{v_p(S)}. \]
For notional convenience, we let $\eta$ denote the fractional expression on the left-hand side, and we write $\mu, \nu$ for its numerator and denominator, respectively.

Recall that $d$ is the squarefree integer for which $d (\Q^{\times})^2=\alpha (\Q^{\times})^2$. Notice that with this choice of $d$, whenever $p$ is a prime for which $v_p(\alpha)=0$, also $v_p(d)=0$ and (if $p$ is odd) $\leg{\alpha}{p} = \leg{d}{p}$.

We evaluate $\leg{d}{\eta}$ in two different ways. On the one hand, since $\leg{\alpha}{q}=-1$,
\[ \leg{d}{\eta} = \leg{d}{q} \prod_{\substack{p\text{ odd} \\ v_p(\alpha)=v_p(\beta)=0 \\ \leg{\alpha}{p}=1}} \leg{d}{p}^{v_p(S)} = \leg{\alpha}{q} \prod_{\substack{p\text{ odd} \\ v_p(\alpha)=v_p(\beta)=0 \\ \leg{\alpha}{p}=1}} \leg{\alpha}{p}^{v_p(S)} = -1. \]

On the other hand, we can show $\leg{d}{\eta}=1$ by verifying the conditions of Lemma \ref{lem:QR}. Let $p$ be an odd prime dividing $d$; then either $v_p(\alpha)< 0$ or $v_p(\alpha) > 0$.  Suppose first that
$v_p(\alpha) < 0$. Then the ultrametric inequality and  \eqref{eq:klarge} imply that $v_p(\mu)=v_p(\nu) = 0$. Write
\begin{multline}\label{eq:munudiff} \nu - \mu = {\bigg(\alpha \prod_{p:~v_p(\alpha)<0} p^{-v_p(\alpha)}\bigg)^{2K} \bigg((\alpha \prod_{p:~v_p(\alpha)<0} p^{-v_p(\alpha)})^{2\ell N n_1} -1\bigg)} \\ + \beta \bigg(\prod_{p:~v_p(\alpha) < 0} p^{-v_p(\alpha)}\bigg)^{2K} \bigg(1-(\prod_{p:~v_p(\alpha) < 0} p^{-v_p(\alpha)})^{2\ell N n_1}\bigg).
\end{multline}
This difference has positive $p$-adic valuation: Indeed, since $p-1\mid M\mid N$,
\[ v_p((\alpha \prod_{p:~v_p(\alpha)<0} p^{-v_p(\alpha)})^{2\ell N n_1}-1)\ge 1 \]
while also (from \eqref{eq:klarge})
\[ v_p\bigg(\beta \bigg(\prod_{p:~v_p(\alpha) < 0} p^{-v_p(\alpha)}\bigg)^{2K}\bigg)  = v_p(\beta) + 2K |v_p(\alpha)| > v_p(\beta) + k > 0. \]
Hence, $v_p(\eta-1) = v_p(\frac\mu\nu -1)=v_p(\mu - \nu) > 0$. Suppose instead that $v_p(\alpha) > 0$. In this case, $v_p(\nu)= v_p(\beta)$. Also,
\[ v_p\bigg(\bigg(\alpha \prod_{p:~v_p(\alpha)<0} p^{-v_p(\alpha)}\bigg)^{2K}\bigg) = 2K v_p(\alpha) > k > v_p(\beta) \]
while
\begin{align*} v_p\bigg(\beta \bigg(1-(\prod_{p:~v_p(\alpha) < 0} p^{-v_p(\alpha)})^{2\ell N n_1}\bigg)\bigg) &= v_p(\beta) + v_p\bigg(1-(\prod_{p:~v_p(\alpha) < 0} p^{-v_p(\alpha)})^{2\ell N n_1}\bigg) \\
&> v_p(\beta), \end{align*}
using for the last inequality that $p-1\mid M \mid N$. From \eqref{eq:munudiff}, $v_p(\mu-\nu) > v_p(\beta)$ and $v_p(\eta-1) = v_p(\mu-\nu) - v_p(\nu) = v_p(\mu-\nu)-v_p(\beta) \ge 1$.

It remains (only) to check that $v_2(\eta-1)\ge 3$. If $v_2(\alpha)> 0$, then $v_2(\nu) = v_2(\beta)$. Moreover,
\[ v_2\bigg(\bigg(\alpha \prod_{p:~v_p(\alpha)<0} p^{-v_p(\alpha)}\bigg)^{2K}\bigg) = 2K v_2(\alpha) > k \ge 3+v_2(\beta)\]
by \eqref{eq:klarge}, while
\begin{align*} v_2\bigg(\beta \bigg(1-(\prod_{p:~v_p(\alpha) < 0} p^{-v_p(\alpha)})^{2\ell N n_1}\bigg)\bigg) &= v_2(\beta) + v_2\bigg(1-(\prod_{p:~v_p(\alpha) < 0} p^{-v_p(\alpha)})^{2\ell N n_1}\bigg) \\
&\ge 3 + v_2(\beta), \end{align*}
using for the last inequality that squares of $2$-adic units are 1 mod $8$. So from \eqref{eq:munudiff}, $v_2(\mu-\nu) \ge 3+v_2(\beta)$ and $v_2(\eta-1) \ge 3$. Now suppose that $v_2(\alpha)< 0$. Then $v_2(\mu) = v_2(\nu) = 0$. Also,
    \[ v_2\bigg((\alpha \prod_{p:~v_p(\alpha)<0} p^{-v_p(\alpha)})^{2\ell N n_1} -1\bigg) \ge 3 \]
while
\[ v_2\bigg(\beta \bigg(\prod_{p:~v_p(\alpha) < 0} p^{-v_p(\alpha)}\bigg)^{2K}\bigg)  = v_2(\beta) + 2K |v_2(\alpha)| > v_2(\beta) + k \ge 3. \]
Hence, $v_2(\eta-1) = v_2(\mu-\nu) \ge 3$ in this case. Finally, suppose that $v_2(\alpha)=0$. Write
\[ \eta = \frac{\alpha^{2K'}-\beta}{\alpha^{2K}-\beta} \prod_{p:~v_p(\alpha) < 0} p^{-2\ell N n_1 v_p(\alpha)}, \]
so that
\begin{multline*} \eta-1 = \left(\frac{\alpha^{2K'}-\beta}{\alpha^{2K}-\beta}-1\right) \left(\prod_{p:~v_p(\alpha) < 0} p^{-2\ell N n_1 v_p(\alpha)}-1\right)\\+\left(\prod_{p:~v_p(\alpha) < 0} p^{-2\ell N n_1 v_p(\alpha)}-1\right) +\left(\frac{\alpha^{2K'}-\beta}{\alpha^{2K}-\beta}-1\right). \end{multline*}
Since $v_2(\prod_{p:~v_p(\alpha) < 0} p^{-2\ell N n_1 v_p(\alpha)}-1) \ge 3$, the desired inequality $v_2(\eta-1)\ge 3$ will follow once it is shown that $v_2(\frac{\alpha^{2K'}-\beta}{\alpha^{2K}-\beta}-1)\ge 3$.  We have
\[ v_2\left(\frac{\alpha^{2K'}-\beta}{\alpha^{2K}-\beta}-1\right) = v_2(\alpha^{2(K'-K)}-1) - v_2(\alpha^{2K}-\beta).
\]
If $v_2(\beta) < 0$ or $v_2(\beta) > 0$, then $v_2(\alpha^{2k}-\beta) = v_2(\beta)$ or $v_2(\alpha^{2k}-\beta)=0$ (respectively); in either case, the inequality $v_2(\frac{\alpha^{2K'}-\beta}{\alpha^{2K}-\beta}-1)\ge 3$ follows from $v_2(\alpha^{2(K'-K)}-1) \ge 3$. So we may suppose that $v_2(\beta)=0$. In this case, we use that $2^{f_2} \mid N$. (Recall that $f_2 = v_2(\alpha^{2K}-\beta)$.) Since
\[ v_2(\alpha^{2\ell n_1}-1) \ge 3, \]
Lemma \ref{lem:LTE} (with $A = \alpha^{2\ell n_1}$) shows that
\[ v_2(\alpha^{2(K'-K)} -1) = v_2(\alpha^{2\ell N n_1}-1) \ge 3 + v_2(N) \ge 3 + f_2, \]
so that
\[ v_2(\alpha^{2(K'-K)}-1) - v_2(\alpha^{2K}-\beta) \ge (3+f_2) - f_2 = 3. \]
This completes the proof of Theorem \ref{thm:extension}.

\section{Proof of Theorem \ref{thm:almostall}}
We start with an overview of the proof, deferring details to \S\S\ref{sec:boundingB}--\ref{sec:part3} below.

Let $N$ be an integer, $N\ge 2$. Let $\mathcal{B}$ denote the collection of nonzero integers $\beta \in [-N,N]$ that \textbf{do not} obey the following hypothesis.

\begin{hyp}\label{hyp:b} There are more than $ \frac{1}{4}\sqrt{N}/\log{N}$ primes $q\in (\frac{1}{2}\sqrt{N},\sqrt{N}]$ with $\leg{\beta}{q}=1$.
\end{hyp}

For large enough $N$, one expects that $\mathcal{B}$ is empty (as would follow from GRH). Proving $\mathcal{B}$ is empty seems quite difficult (see  \cite{DKSZ20} for a discussion of related problems), but for our purposes it is enough to know that $\mathcal{B}$ is fairly small. We prove in \S\ref{sec:boundingB} that
\begin{equation}\label{eq:bbound} \#\mathcal{B} \ll N^{1/2}\log{N}. \end{equation}

Now fix an integer $\beta$ with $1 < |\beta| \le N$ and $\beta\notin \mathcal{B}$. For each integer $\alpha$, let $\mathcal{P}(\alpha)$ denote the set of odd primes $q \in (\frac{1}{2}\sqrt{N}, \sqrt{N}]$ for which
\begin{enumerate}
    \item $\leg{\beta}{q}=1$,
    \item $\alpha$ is a primitive root modulo $q$.
\end{enumerate}

If there is any $q \in \mathcal{P}(\alpha)$ for which $v_q(\alpha^{q-1}-1) = 1$, we can deduce from Theorem \ref{thm:extension} (applied with this $q$) that $\alpha,\beta$ is special.  Indeed, conditions (i) and (ii) of Theorem \ref{thm:extension} are obvious. Since $\alpha^2$ generates the subgroup of nonzero squares mod $q$, we can choose $k$ with $\alpha^{2k}\equiv \beta \pmod{q}$. Then $v_q(\alpha^{2k}-\beta) \ge 1 = v_q(\alpha^{q-1}-1)$, and so we have (iii).

In \S\S\ref{sec:part2}--\ref{sec:part3} we show there is a $q \in \mathcal{P}(\alpha)$ with $v_q(\alpha^{q-1}-1) = 1$ for every $\alpha \in [-N,N]$, apart from at most $O(N^{1/2}\log{N})$ exceptions. Putting this together with our bound \eqref{eq:bbound} on $\#\mathcal{B}$, we find that the number of non-special integer pairs $\alpha,\beta$ with $1 < |\alpha|, |\beta| \le N$ is
\[ \ll N \cdot N^{1/2}\log{N} + N^{1/2}\log{N} \cdot N \ll N^{3/2}\log{N}, \]
as claimed in Theorem \ref{thm:almostall}.

\subsection{Bounding $\#\mathcal{B}$}\label{sec:boundingB} Let $Q = \pi(\sqrt{N}) - \pi(\frac{1}{2}\sqrt{N})$, so that $Q \sim \sqrt{N}/\log{N}$ as $N\to\infty$. For each integer $\beta$, put $S_{\beta} = \sum_{q \in (\frac{1}{2}\sqrt{N},\sqrt{N}]} \leg{\beta}{q}$.
If $\beta\in \mathcal{B}$, then (assuming $N$ is large)
\begin{align*} S_{\beta} + Q =& \sum_{q \in (\frac{1}{2}\sqrt{N},\sqrt{N}]} \left(1+\leg{\beta}{q}\right) \le 2 \cdot \frac{1}{4}\sqrt{N}/\log{N} + \#\{q \in (\frac{1}{2}\sqrt{N}, \sqrt{N}]: q \mid \beta\} \\
&\le \frac{1}{2}\sqrt{N}/\log{N} + 2 < \frac{2}{3}Q,
\end{align*}
so that $S_{\beta} < -\frac{1}{3}Q$. It follows that
\begin{equation}\label{eq:2ndmoment} \sum_{|\beta| \le N} S_{\beta}^2 \gg Q^2  \#\mathcal{B}. \end{equation}
On the other hand, $\sum_{|\beta|\le N} S_{\beta}^2 = S_1 + S_2$, where
\[ S_1 = \sum_{q\in (\frac{1}{2}\sqrt{N},\sqrt{N}]} \sum_{|\beta| \le N} \leg{\beta}{q}^2,\quad\text{and}\quad S_2 = \sum_{|\beta| \le N} \sum_{\substack{q_1, q_2 \in (\frac{1}{2}\sqrt{N},\sqrt{N}] \\ q_1 \ne q_2}}  \leg{\beta}{q_1 q_2}.\]
When $q_1\ne q_2$, the Jacobi symbol $\leg{\cdot}{q_1 q_2}$ is a nontrivial character mod $q_1 q_2$, and the P\'olya--Vinogradov inequality yields
\[ S_2 = \sum_{\substack{q_1, q_2 \in (\frac{1}{2}\sqrt{N},\sqrt{N}] \\ q_1 \ne q_2}} \sum_{|\beta| \le N} \leg{\beta}{q_1 q_2} \ll \sum_{\substack{q_1, q_2 \in (\frac{1}{2}\sqrt{N},\sqrt{N}] \\ q_1 \ne q_2}} \sqrt{q_1 q_2} \log(q_1 q_2) \ll Q^2 N^{1/2}\log{N}. \]
Trivially, $S_1 \ll QN$, and a straightforward check shows $QN \ll Q^{2} N^{1/2}\log{N}$. Hence, $\sum_{|\beta| \le N} S_{\beta}^2 \ll Q^2 N^{1/2}\log{N}$. The desired bound $\#\mathcal{B} \ll N^{1/2}\log{N}$ follows by comparison with \eqref{eq:2ndmoment}.

\subsection{Estimating $\#\mathcal{P}(\alpha)$}\label{sec:part2}  For the remainder of the proof, $\beta$  is a fixed integer with $1< |\beta| \le N$ and $\beta\notin\mathcal{B}$.

Recall that $\mathcal{P}(\alpha)$ denotes the set of primes $q \in (\frac{1}{2}\sqrt{N},N]$ for which $\leg{\beta}{q}=1$ and $\alpha$ is  a primitive root mod $q$. Let us argue that there is an absolute constant $c_1 > 0$ such that
\[ \#\mathcal{P}(\alpha) > c_1 \sqrt{N}/\log{N} \]
for all but $O(N^{1/2} \log{N})$ integers $\alpha \in [-N,N]$.



Write $\mathcal{P}_0$ for the set of primes $q\in (\frac{1}{2}\sqrt{N},\sqrt{N}]$ with $\leg{\beta}{q}=1$. Since $\beta\notin\mathcal{B}$, we have $\#\mathcal{P}_0 > \frac{1}{4}\sqrt{N}/\log{N}$. It will be helpful momentarily to have a lower bound on $\sum_{q \in \mathcal{P}_0} \phi(q-1)/q$. For this we borrow an argument of Gallagher \cite[p. 17]{gallagher67}: By Cauchy--Schwarz,
\[ \sum_{q \in \mathcal{P}_0} \frac{\phi(q-1)}{q} \sum_{q \in \mathcal{P}_0} \frac{q}{\phi(q-1)} \ge (\#\mathcal{P}_0)^2  > \frac{1}{16} N/(\log{N})^2,\]
while (see eq.\ (10) of \cite{gallagher67}) for a certain absolute positive constant $c_2$,
\[ \sum_{q \in \mathcal{P}_0} \frac{q}{\phi(q-1)} \le 2\sum_{q \le \sqrt{N}} \frac{q-1}{\phi(q-1)} \le c_2 \sqrt{N}/\log{N}. \]
Hence, $\sum_{q \in \mathcal{P}_0} \phi(q-1)/q \ge c_3 \sqrt{N}/\log{N}$ for an absolute constant $c_3 > 0$.


Let $\mathcal{N}$ be the set of integers $\alpha \in [-N,N]$ with $\#\mathcal{P}(\alpha) \le \frac{1}{2}c_3 \sqrt{N}/\log{N}$. We apply the large sieve (Proposition \ref{prop:largesieve}) to estimate $Z = \#\mathcal{N}$. Observe that
\begin{align*} \sum_{q \in \mathcal{P}_0} \sum_{\substack{h\bmod{q} \\ h \text{ generates $\mathbb{F}_q^{\times}$}}} Z(q,h) &=
\sum_{q \in \mathcal{P}_0} \sum_{\alpha \in \mathcal{N}} 1_{\alpha\text{ prim. root mod $q$}} \\
&= \sum_{\alpha \in \mathcal{N}} \#\mathcal{P}(\alpha) \le \frac{1}{2}c_3 Z \sqrt{N}/\log{N}.
\end{align*}
On the other hand,
\[ \sum_{q \in \mathcal{P}_0} \sum_{\substack{h\bmod{q} \\ h \text{ generates $\mathbb{F}_q^{\times}$}}} \frac{Z}{q} = Z \sum_{q \in \mathcal{P}_0} \frac{\phi(q-1)}{q} \ge c_3 Z\sqrt{N}/\log{N}.\]
Comparing the last two displays, and using the triangle inequality,
\[ \sum_{q \in \mathcal{P}_0} \sum_{h\bmod{q}} \ \left|Z(q,h)-\frac{Z}{q}\right| \ge \frac{1}{2} c_3 Z\sqrt{N}/\log{N}. \]
Applying Cauchy--Schwarz (viewing $|Z(q,h)-Z/q| = q^{-1/2} \cdot q^{1/2} |Z(q,h)-Z/q|)$,
\[ \left(\sum_{q \in \mathcal{P}_0} \sum_{h\bmod{q}} \ \left|Z(q,h)-\frac{Z}{q}\right|\right)^2 \le \#\mathcal{P}_0 \cdot \sum_{q\in \mathcal{P}_0} q \sum_{h\bmod{q}} \left(Z(q,h)-\frac{Z}{q}\right)^2.\]
Since $\#\mathcal{P}_0 \le \pi(\sqrt{N})$, we conclude that
\[ \sum_{q \in \mathcal{P}_0}q  \sum_{h\bmod{q}}\left(Z(q,h)-\frac{Z}{q}\right)^2 \ge \frac{(\frac{1}{2} c_3 Z\sqrt{N}/\log{N})^2}{\pi(\sqrt{N})} \gg Z^2 \sqrt{N}/\log{N}. \]
On the other hand, Proposition \ref{prop:largesieve} implies that the left-hand side of the last display is $O(NZ)$. Thus, $\#\mathcal{N} = Z  \ll N^{1/2} \log{N}$.

We have shown what was claimed at the start of this section, with the constant $c_1 = \frac{1}{2}c_3$: For all but $O(N^{1/2}\log{N})$ integers $\alpha \in [-N,N]$, we have $\#\mathcal{P}(\alpha) > \frac{1}{2}c_3 \sqrt{N}/\log{N}$.

\subsection{Finding $q$ with $v_q(\alpha^{q-1}-1) = 1$}\label{sec:part3} Let $\mathcal{E}$ be the collection of integers $\alpha \in [-N,N]$ such that $v_q(\alpha^{q-1}-1)>1$ for all $q \in \mathcal{P}(\alpha)$. To complete the proof of Theorem \ref{thm:almostall}, it is enough to show that $\#\mathcal{E} \ll N^{1/2}\log{N}$.

Let $\mathcal{E}'$ be the subset of $\mathcal{E}$ where we keep only those $\alpha$ with $\#\mathcal{P}(\alpha) > \frac{1}{2}c_3 \sqrt{N}/\log{N}$. From our work in the last section, we need only show $\#\mathcal{E'} \ll N^{1/2}\log{N}$. In fact, we will show that $\#\mathcal{E}' \ll N^{1/2}$.

For each $\alpha \in \mathcal{E}'$ and each $q \in \mathcal{P}(\alpha)$, put
\[ \mathcal{A}_q(\alpha) = \{\alpha + jq: 0 < |j| < q/2\} \cap [-N,N]. \]
Let us make a few observations about the sets $\mathcal{A}_q(\alpha)$.

First, $\#\mathcal{A}_q(\alpha) \ge \frac{q-1}{2} > \frac{1}{5}\sqrt{N}$ (for large $N$). Indeed, if $\alpha > 0$, then $\alpha +jq \in [-N,N]$ whenever $-1 \ge j \ge -q/2$, while if $\alpha < 0$, then $\alpha+jq \in [-N,N]$ when $1 \le j \le q/2$.

Second, if we fix $\alpha \in \mathcal{E'}$, then the sets $\mathcal{A}_q(\alpha)$ are disjoint for distinct $q\in \mathcal{P}(\alpha)$. Otherwise, we are led to an equation $\alpha + jq = \alpha + j'q'$ where $q,q' \in \mathcal{P}(\alpha)$, $q\ne q'$. This  forces $q$ to divide $j'$, which is impossible since $0<|j'| < \frac{1}{2}q' \le \frac{1}{2}\sqrt{N} < q$.

Third, suppose that $q \in \mathcal{P}(\alpha) \cap \mathcal{P}(\alpha')$, where $\alpha$ and $\alpha'$ are distinct elements of $\mathcal{E'}$. Then $\mathcal{A}_q(\alpha)$ and $\mathcal{A}_q(\alpha')$ are disjoint. Otherwise, $\alpha = \alpha' + Jq$ for some integer $J\not\equiv 0\pmod{q}$. Then
\[ \alpha^{q-1} \equiv \alpha'^{q-1} + J(q-1) q \alpha'^{q-2} \pmod{q^2}, \]
contradicting that $\alpha^{q-1}\equiv \alpha'^{q-1}\equiv 1\pmod{q^2}$.

Let $\mathcal{A}(\alpha) = \cup_{q \in \mathcal{P}(\alpha)} \mathcal{A}_{q}(\alpha)$. By the first and second observations,
\[ \sum_{\alpha \in \mathcal{E}'} \#\mathcal{A}(\alpha) = \sum_{\alpha \in \mathcal{E}'} \sum_{q \in \mathcal{P}(\alpha)} \#\mathcal{A}_q(\alpha) \ge \#\mathcal{E}' \cdot \#\mathcal{P}(\alpha) \cdot \frac{1}{5}\sqrt{N} > \#\mathcal{E}' \cdot \frac{1}{10}c_3  \cdot N/\log{N}. \]
On the other hand, since each $\mathcal{A}(\alpha)\subset [-N,N]$, the second and third observations imply that
\begin{align*} \sum_{\alpha \in \mathcal{E'}} \#\mathcal{A}(\alpha) &= \sum_{\substack{|n| \le N}} \sum_{\substack{\alpha \in \mathcal{E}'}} 1_{n \in \mathcal{A}(\alpha)} \\&=  \sum_{|n| \le N} \sum_{\substack{\alpha \in \mathcal{E}'}} \sum_{q \in \mathcal{P}(\alpha)} 1_{n \in \mathcal{A}_q(\alpha)} \\
&= \sum_{|n| \le N} \sum_{q \in (\frac{1}{2}\sqrt{N},\sqrt{N}]} \sum_{\substack{\alpha \in \mathcal{E}' \\ q \in \mathcal{P}(\alpha)}} 1_{n \in \mathcal{A}_q(\alpha)} \\&\le \sum_{|n| \le N} \sum_{q \in (\frac{1}{2}\sqrt{N},\sqrt{N}]} 1 \ll N^{3/2}/\log{N}. \end{align*}
(Here the second observation has been used to go from the first line to the second, while the third observation has been used to go from the third line to the fourth.) Comparing the last two displays gives $\#\mathcal{E}' \ll N^{1/2}$, as claimed. This completes the proof of Theorem \ref{thm:almostall}.

\begin{rmk} For any \emph{fixed} integer $\beta$ with $|\beta| > 1$, Hypothesis \ref{hyp:b} holds for all large $N$, as a consequence of the Chebotarev density theorem applied to $\Q(\sqrt{\beta})$. (Alternatively, one can apply the prime number theorem for arithmetic progressions.) Following the above arguments, we conclude that $\alpha,\beta$ is special for all but $O(N^{1/2}\log{N})$ integers $\alpha$ with $|\alpha| \le N$. \end{rmk}

\section*{Acknowledgements}
The authors would like to express their thanks to Alexandr Kalmynin for pointing out the paper \cite{DKSZ20} and to the referee for helpful comments. The preparation of this paper was  possible only due to the Nineteenth Annual Workshop on Combinatorial and Additive Number Theory (CANT 2021, May 24--28, 2021). Both authors are very grateful to Melvyn Nathanson for organizing this beautiful conference. The work of Sergei Konyagin was performed at the Steklov International Mathematical Center
and supported by the Ministry of Science and Higher Education of
the Russian Federation (agreement no. 075-15-2019-1614).
Paul Pollack is supported by the US National Science Foundation (award DMS-2001581).

\providecommand{\bysame}{\leavevmode\hbox to3em{\hrulefill}\thinspace}
\providecommand{\MR}{\relax\ifhmode\unskip\space\fi MR }
\providecommand{\MRhref}[2]{%
  \href{http://www.ams.org/mathscinet-getitem?mr=#1}{#2}
}
\providecommand{\href}[2]{#2}

\end{document}